\documentclass[11pt]{article}
\usepackage{currfile}
\usepackage{color}
\usepackage{mathrsfs}
\usepackage{graphicx}
\usepackage{caption}
\usepackage{subfigure}
\usepackage{amsmath}
\usepackage{amssymb}
\usepackage{setspace}
\usepackage{epstopdf}

\newtheorem{thm}{Theorem}[section]

\newtheorem{lem}[thm]{Lemma}
\newtheorem{prop}[thm]{Proposition}
\newtheorem{prob}[thm]{Problem}
\newenvironment{proof}{\noindent {\bf
Proof.}}{\rule{3mm}{3mm}\par\medskip}

\def\marker{\>\hbox{${\vcenter{\vbox{
    \hrule height 0.4pt\hbox{\vrule width 0.4pt height 6pt
    \kern6pt\vrule width 0.4pt}\hrule height 0.4pt}}}$}\>}

\def\qed{ \hfill $\square$}

\newcommand{\es}{{\rm es}_{\chi}}

\title{On critical graphs for the { {chromatic edge-stability number}}}
\author{\small {Hui Lei$^1$, Xiaopan Lian$^2$, Xianhao Meng$^3$, Yongtang Shi$^2$, Yiqiao Wang$^4$\thanks{The corresponding author.}}\\
{\small $^1$ School of Statistics and Data Science, LPMC and KLMDASR}\\
{\small Nankai University, Tianjin 300071, China}\\
{\small $^2$ Center for Combinatorics and LPMC}\\
{\small Nankai University, Tianjin 300071, China}\\
{\small $^3$ College of Software, Nankai University, Tianjin 300350, China}\\
{\small $^4$ School of Management}\\
{\small Beijing University of Chinese Medicine, Beijing 100029, China}\\
{\small Email: hlei@nankai.edu.cn; xiaopanlian@mail.nankai.edu.cn;} \\ {\small mm17862903862@163.com; shi@nankai.edu.cn; yqwang@bucm.edu.cn}\\
}
\date{\today}

\begin{document}

\maketitle
\begin{abstract}
The { {\em chromatic edge-stability number}} $es_{\chi}(G)$ of a graph $G$ is the minimum number of edges whose removal results in a spanning subgraph with the chromatic number smaller than that of $G$. A graph $G$ is called {\em $(3,2)$-critical} if $\chi(G)=3$, $es_{\chi}(G)=2$ and for any edge $e\in E(G)$, $es_{\chi}(G-e)<es_{\chi}(G)$. In this paper, we characterize $(3,2)$-critical graphs which contain at least five odd cycles. This answers a question proposed by
 Bre\v{s}ar, Klav\v{z}ar and Movarraei in [Critical graphs for the chromatic edge-stability number, {\it Discrete Math.} {\bf 343}(2020) 111845].\\


\end{abstract}

\section{Introduction}

Let $G=(V(G),E(G))$ be a graph. A function $c:V(G)\rightarrow[k]=\{1,\ldots,k\}$ is called
a {\em proper coloring} of $G$, if $c(u)\neq c(v)$ for any $uv\in E(G)$. The {\em chromatic number} of $G$, denoted by $\chi(G)$, is the smallest integer $k$ such that $G$ admits a proper coloring using $k$ colors. The { {\em chromatic edge-stability number}} of $G$, denoted by $\es(G)$, is the minimum number of edges of $G$ whose removal results in a graph with the chromatic number smaller than that of $G$. The chromatic edge-stability number was first studied by Staton~\cite{14}, which provided upper bounds of $es_{\chi}$ for regular graphs in terms of the size of a given graph. The invariant was subsequently  investigated in~\cite{arumugam-2008,2,8}.
For a graph $G$ with $\chi(G)=3$, the chromatic edge-stability number is equal to the bipartite edge frustration \cite{4}, which is defined as the smallest number of edges that have to be deleted from $G$ to obtain a bipartite spanning subgraph.

For any $u,v\in V(G)$, let $d_G(u,v)$ denote the length of the shortest $(u,v)$-path. For any $A\subseteq E(G)$, let $G-A$ be the graph obtained from $G$ by deleting all the edges in $A$. If $A=\{e\}$, we simply write $G-e$ instead of $G-\{e\}$. We say a graph  $G$ is {\em edge-stability critical} if $es_{\chi}(G-e)<es_{\chi}(G)$ holds for every edge $e\in E(G)$. A graph $G$ is called {\em $(k, \ell)$-critical}, if $G$ is an edge-stability critical graph with $\chi(G)=k$ and $es_{\chi}(G)=\ell$, for $k,\ell\geq2$. Naturally, a graph $G$ is $(k,2)$-critical if and only if for every edge $e\in  E(G)$, $\chi(G-e)=k$, and there exists an edge $f\in E(G-e)$ such that $\chi(G-\{e,f\})=k-1$. In this paper we focus on $(3,2)$-critical graphs and the graphs we consider are simple.

In \cite{1}, the authors proved the following theorem.
\begin{thm}[\cite{1}]
$\mathcal{A}\cup \mathcal{B}\cup\mathcal{C}\cup\mathcal{D}$ is the family of $(3,2)$-critical graphs (without isolated vertices) that contain at most four odd cycles. \end{thm}
These four graph families are defined as follows. Let $G+H$ denote the disjoint union of graphs $G$ and $H$. We use $C_n$ to denote the cycle on $n$ vertices. A path or a cycle is odd if it has an odd number of edges, otherwise, we say it even. Then the families of $(3,2)$-critical graphs mentioned in \cite{1} are as follows. Let $\mathcal{A}=\{C_{2k+1}+C_{2\ell+1}\mid k, \ell\geq 1\}$ and let $\mathcal{B}$ be the family of graphs that are obtained from $C_{2k+1}+C_{2\ell+1},k, \ell\geq 1$, by identifying a vertex of $C_{2k+1}$ and a vertex of $C_{2\ell+1}$. Let $x_i$, $y_i$ be the end vertices of the paths $Q_i,i\in[4]$, exactly two of the $Q_i$ are odd, and at most one of them is of length one. The family $\mathcal{C}$ consists of the graphs that are obtained from such four paths, by identifying the vertices $x_1$, $x_2$, $x_3$, and $x_4$ and also identifying the vertices $y_1$, $y_2$, $y_3$, and $y_4$. The family $\mathcal{D}$ consists of the following subdivisions of the graph $K_4$: (i) all the subdivided paths are of odd length, (ii) exactly three of the paths are odd, and these three paths induce an odd cycle or a path, (iii) exactly two of the paths are odd, and these two paths are vertex disjoint, and (iv) exactly two of the paths are even and these two paths have a common vertex.
%

At the end of \cite{1}, Bre\v{s}ar, Klav\v{z}ar, and Movarraei defined the family $\mathcal{E}$, which is obtained from the disjoint union of $k$ even cycles $C_{2n_1},\ldots, C_{2n_k}$ as follows. For each $i\in[k]$, let $x_i$ and $y_i$ be any two distinct vertices of $C_{2n_i}$, where they only require that $\sum_{i=1}^kd_{C_{2n_i}}(x_i,y_i)$ is odd. A graph $G\in \mathcal{E}$ is obtained by identifying $y_i$ and $x_{i+1}$ for $i\in[k-1]$, and identifying $y_k$ and $x_1$. They proposed the following problem and suspected it has a positive answer.

\begin{prob}[\cite{1}]
Is it true that $\mathcal{A}\cup \mathcal{B}\cup\mathcal{C}\cup\mathcal{D}\cup \mathcal{E}$ is the family of $(3,2)$-critical graphs (without isolated vertices)?
\end{prob}

We answer this problem by giving a positive proof.

\begin{thm}\label{mainthm}
$\mathcal{E}$ is the family of $(3,2)$-critical graphs (without isolated vertices) which contain at least five odd cycles.
\end{thm}

\section{Properties of $(3, 2)$-critical graphs}
In this section,  we establish some structural results on $(3, 2)$-critical graphs.
The following lemmas and propositions
were proved in \cite{BM,1} and will be used in this paper.
\begin{lem}[\cite{1}]\label{lemma1}
If $G$ is a $(3, 2)$-critical graph that contains at least three odd cycles, then every two distinct odd cycles intersect
in more than one vertex.
\end{lem}


Let $\mathcal{G}_i$ ($i\in[7]$) be the family of graphs as shown in Figure \ref{fig15}. For $i\in[7]$ and $G_i\in\mathcal{G}_i$, three internally disjoint $(x, y)$-paths of $G_i$ formed by solid lines from left to right are denoted as $Q_1,Q_2,Q_3$. Let $D_1=Q_1\cup Q_2$ and $D_2=Q_2\cup Q_3$. Every graph in $\mathcal{G}_i$ ($i\in[7]$) satisfies that $D_1$ and $D_2$ are odd cycles, and the dotted line is internally disjoint from these solid lines. Let $\mathcal{G}=\bigcup_{i\in[7]}\mathcal{G}_i$.
\begin{prop}[\cite{1}]\label{prop}
If $G$ is a $(3,2)$-critical graph that contains at least three odd cycles, then there exists an $H\in\mathcal{G}$ such that $H\subseteq G$.

\end{prop}

\begin{figure}[htbp]
  \begin{center}
  \includegraphics[scale=0.41]{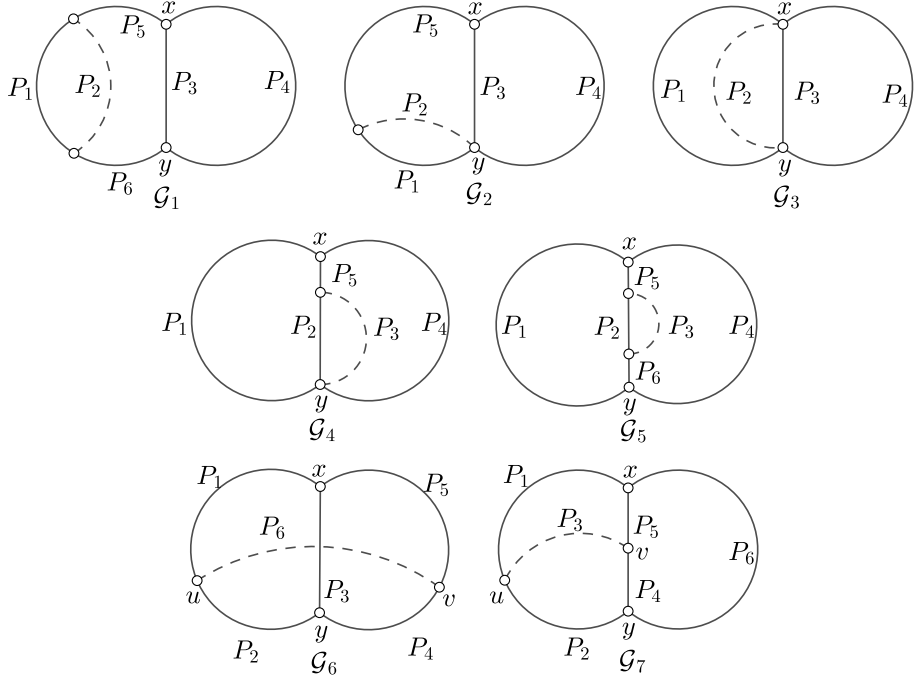}
  \caption{ Seven families of subgraphs of (3,2)-critical graphs.}\label{fig15}
  \end{center}
  \end{figure}

A graph $G$ is {\em connected} if there is a $(u,v)$-path in $G$ for any $u,v\in V(G)$. A {\em separation} of a connected graph is a decomposition of the graph into two
nonempty connected subgraphs which have just one vertex in common. This common vertex is called a {\em separating vertex} of the graph.
A graph is {\em nonseparable} if it is connected and has no separating vertices. Let $F$ be a  nontrivial proper subgraph of a graph $G$. An {\em ear} of $F$ in $G$ is a nontrivial path in $G$ whose endpoints lie in $F$ but whose internal vertices do not. An ear is an {\em open ear} if the endpoints of the path are distinct.
For completeness, we present the proof of the following proposition from \cite{BM}.

\begin{prop}[\cite{BM}]\label{ear}
Let $F$ be a nontrivial proper subgraph of a nonseparable graph $G$. Then $F$ has an open ear in $G$.
\end{prop}
\begin{proof}
If $F$ is a spanning subgraph of $G$, then $E(G)\setminus E(F)$ is nonempty because,
by hypothesis, $F$ is a proper subgraph of $G$. Any edge in $E(G)\setminus E(F)$ is then an
ear of $F$ in $G$. We may suppose, therefore, that $F$ is not spanning.
Since $G$ is connected, there is an edge $xy$ of $G$ with $x\in V(F)$ and $y\in V(G)\setminus V(F)$. Because $G$ is nonseparable, $G-x$ is connected. So there is a $(y,F-x)$-path
$Q$ in $G-x$. The path $P:= xyQ$ is an open ear of $F$.
\end{proof}

We first prove  the following lemma.
\begin{lem}\label{conn}
If $G$ is a $(3,2)$-critical graph that contains at least three odd cycles (without isolated vertices), then $G$ is nonseparable.
\end{lem}
\begin{proof} We claim that if $G$ is $(3,2)$-critical, then every edge of $G$ is contained in at least one odd cycle. Suppose $e\in E(G)$ and $e$ is not contained in any odd cycle. By the definition of $(3,2)$-critical graph, there exists at least one edge $f\in E(G)\setminus \{e\}$ such that $\chi(G-\{e,f\})=2$. Since $e$ is not contained in any odd cycle, we have $\chi(G-f)=2$, contradicting the fact that $G$ is $(3,2)$-critical.

Let $G$ be a $(3,2)$-critical graph that contains at least three odd cycles (without isolated vertices). Suppose $G$ is not a connected graph or $G$ contains a separating vertex $v$. Let $G=G_1\cup G_2$ with $G_1\cap G_2=\emptyset$ or $\{v\}$, and there is at least one edge in $G_i$ ($i\in[2]$). By Lemma~\ref{lemma1}, one of $G_1$ and $G_2$ contains all odd cycles. Thus there exists at least one edge that is not contained in any odd cycle, a contradiction. Hence, $G$ is nonseparable.
\end{proof}

{ For an edge $e_i$, let $\mathcal{F}_i=\{f_i\in E(G)\mid \chi(G-\{e_i,f_i\})=2\}$}.
\begin {thm}\label{thm1}
Let $G$ be a $(3,2)$-critical graph with at least three odd cycles.  Suppose there are two odd cycles $D_1$ and $D_2$ in $G$ satisfying the following three conditions.
\begin{description}
\item[(1)] The intersection of $D_1$ and $D_2$ is a nontrivial path;
\item[(2)] {  There are two edges $e_1$ and $e_2$ in $G$ such that }$e_1\in E(D_1)\backslash E(D_2)$ and $e_2\notin E(D_2)$;
\item[(3)] $\mathcal{F}_1\subseteq E(D_1)\cap E(D_2)$ and $\mathcal{F}_{2}\cap (E(D_1)\cap E(D_2))\neq \emptyset$.
\end{description}
Then $\mathcal{F}_1\subseteq \mathcal{F}_2$. In particular, if $e_2\in E(D_1)\backslash E(D_2)$ and $\mathcal{F}_2\subseteq E(D_1)\cap E(D_2)$, then $\mathcal{F}_1=\mathcal{F}_2$.
\end{thm}

\begin{figure}[htbp]
\begin{center}
  \includegraphics[scale=0.43]{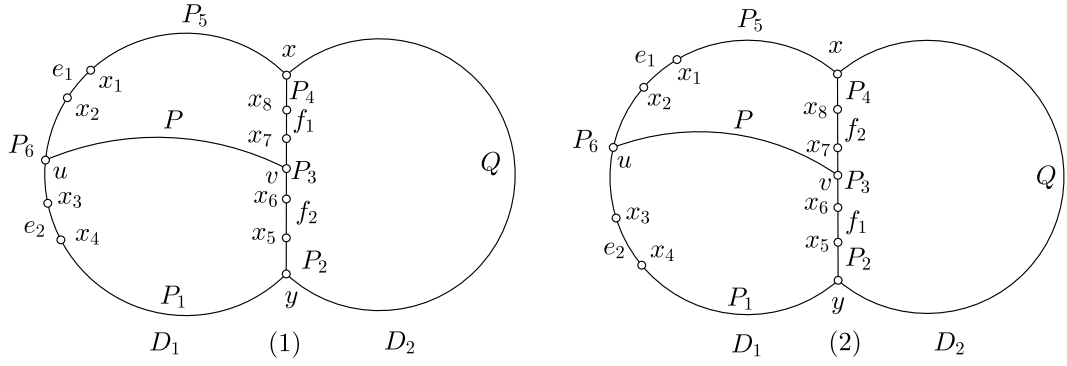}
\caption{ Subgraphs of $G$.}\label{fig6}
\end{center}
\end{figure}

\noindent{\bf Proof.}
Since $G$ is $(3,2)$-critical, $\mathcal{F}_1$ and $\mathcal{F}_2$ are non-empty. Suppose that there are two odd cycles $D_1$ and $D_2$ in $G$ satisfying the above three conditions, but $\mathcal{F}_1\nsubseteq\mathcal{F}_2$. Then there exists an edge $f_1\in \mathcal{F}_1\setminus \mathcal{F}_2$ such that $\chi(G-\{e_2,f_1\})=3$. This implies that there exists at least one odd cycle $C$  which is distinct from $D_1$ and $D_2$ in $G$ such that $e_2,f_1\notin E(C)$, since $G$ is a $(3,2)$-critical graph with at least three odd cycles.  Moreover, we have $(\{e_1\}\cup \mathcal{F}_2)\subseteq E(C)$ since $e_2,f_1\notin E(C)$. Next we show that this will lead to a contradiction.

Denote by $x$ and $y$ the two endpoints of the path which is the intersection of $D_1$ and $D_2$.
Suppose $e_1=x_1x_2$ and $e_2=ab$. {  Since $\mathcal{F}_{2}\cap (E(D_1)\cap E(D_2))\neq \emptyset$}, let $f_2\in \mathcal{F}_{2}\cap (E(D_1)\cap E(D_2))$. Since $\chi(G-\{e_2,f_1\})=3$, we have $f_2\neq f_1$. Let $x_5,x_6,x_7,x_8$
be the endpoints of $f_1$ and $f_2$ of which assignment depends on their position in $D_1\cap D_2$ (see Figure~\ref{fig6}, where each of the two cases has a separate figure). Let $P_2,P_3,P_4,P_5$ be the $(y,x_5)$-path, $(x_6,x_7)$-path, $(x_8,x)$-path, $(x,x_1)$-path of $D_1$ and $Q$ be the $(y,x)$-path of $D_2$ in a counter clockwise direction, respectively, as shown in Figure~\ref{fig6}. If $e_2\in E(D_1)$, then let $e_2=ab=x_3x_4$ and $P_1,P_6$ be the $(x_4,y)$-path, $(x_2,x_3)$-path of $D_1$ in a counter clockwise direction, respectively.
If $e_2\notin E(D_1)$, then let $P_6=P_1$ be the $(x_2,y)$-path of $D_1$ in a counter clockwise direction and $x_2=x_3=x_4$.
We first prove the following claim.\medskip

\noindent  {\bf Claim}: Let  $u\in (V(D_1)\cup V(D_2))\setminus V(P_3)$ and $v\in V(P_3)$.  Then there is no $(u,v)$-path $P$ such that  $V(P)\cap (V(D_1)\cup V(D_2))=\{u,v\}$, where $P\neq f_1, f_2$.\medskip

\begin{proof} Suppose there is a $(u,v)$-path $P$ such that  $V(P)\cap (V(D_1)\cup V(D_2))=\{u,v\}$ and $P\neq f_1, f_2$. It suffices to consider the two structures as shown in Figure~\ref{fig6}. The difference between two graphs in Figure~\ref{fig6} is the position relation of the three edge $e_1$, $f_1$ and $f_2$. In the following, we will consider the two structures simultaneously. Let $M_1$ and $M_2$ be the $(x_6,v)$-path and $(v,x_7)$-path of $D_1$ in a counter clockwise direction, respectively.

%
First, suppose $u\in V(P_1)$. Let $M_3$ and $M_4$ be the $(x_4,u)$-path and $(u,y)$-path of $D_1$ in a counter clockwise direction, respectively. { If $|E(P\cup M_4)|$ and $|E(M_1\cup x_5x_6\cup P_2)|$ have different parity, then $P\cup M_4\cup P_2\cup x_5x_6\cup M_1$ is an odd cycle. Otherwise} $P\cup M_4\cup Q\cup P_4\cup x_8x_7\cup M_2$ is an odd cycle since $D_2$ is an odd cycle. If $P\cup M_4\cup P_2\cup x_5x_6\cup M_1$ is an odd cycle, then $\chi(G-\{e_1,f_1\})=3$ in Figure~\ref{fig6} (1)  and  $\chi(G-\{e_2,f_2\})=3$ in Figure~\ref{fig6} (2). If $P\cup M_4\cup Q\cup P_4\cup x_8x_7\cup M_2$ is an odd cycle, then $\chi(G-\{e_2,f_2\})=3$ in Figure~\ref{fig6} (1)  and  $\chi(G-\{e_1,f_1\})=3$ in Figure~\ref{fig6} (2). This contradicts $\chi(G-\{e_i,f_i\})=2$ for $i\in [2]$. Similarly, if $u\in V(P_5)$, then we also can get a contradiction.

{ Now suppose $u\in V(P_2)$. Let $M_5$ and $M_6$ be the $(y,u)$-path and $(u,x_5)$-path of $D_1$ in a counter clockwise direction, respectively. If $|E(P)|$ and $|E(M_1\cup x_6x_5\cup M_6)|$ have different parity, then $P\cup M_1\cup x_6x_5\cup M_6$ is an odd cycle. Otherwise  $M_5\cup P\cup M_2\cup x_7x_8\cup P_4\cup Q$ is an odd cycle since $D_2$ is an odd cycle. If $P\cup M_1\cup x_6x_5\cup M_6$ is an odd cycle, then $\chi(G-\{e_1,f_1\})=3$ in Figure~\ref{fig6} (1)  and  $\chi(G-\{e_2,f_2\})=3$ in Figure~\ref{fig6} (2). If $M_5\cup P\cup M_2\cup x_7x_8\cup P_4\cup Q$ is an odd cycle, then $\chi(G-\{e_2,f_2\})=3$ in Figure~\ref{fig6} (1)  and  $\chi(G-\{e_1,f_1\})=3$ in Figure~\ref{fig6} (2), a contradiction. Similarly, if $u\in V(P_4\cup Q)$, then we also can get a contradiction.}

Finally, we only need to consider the case of $e_2\in E(D_1)$ and $u\in V(P_6)$. In this case, let { $M_7$ and $M_8$} be the $(x_2,u)$-path and $(u,x_3)$-path of $D_1$ in a counter clockwise direction, respectively. In Figure~\ref{fig6} (1), since $D_1$ is an odd cycle, either $P\cup M_7\cup e_1\cup P_5\cup P_4\cup f_1\cup M_2$ or $P\cup M_8\cup e_2\cup P_1\cup P_2\cup f_2\cup M_1$ is an odd cycle. Then $\chi(G-\{e_2,f_2\})=3$ or $\chi(G-\{e_1,f_1\})=3$, a contradiction. In Figure~\ref{fig6} (2), { since $D_1$ and $D_2$ are odd cycles, we have either $|E(P\cup M_8\cup e_2\cup P_1)|$ and $|E(P_2\cup f_1\cup M_1)|$ have the same parity or $|E(P\cup M_7\cup e_1\cup P_5)|$ and $|E(P_4\cup f_2\cup M_2)|$ have the same parity. So either} $P\cup M_8\cup e_2\cup P_1\cup Q\cup P_4\cup f_2\cup M_2$ or $P\cup M_7\cup e_1\cup P_5 \cup Q\cup P_2\cup f_1\cup M_1$ is an odd cycle. Then $\chi(G-\{e_1,f_1\})=3$ or $\chi(G-\{e_2,f_2\})=3$, a contradiction. Hence the claim holds.\qed\\

Let $w=x_6$ if $f_2=x_5x_6$ and $w=x_7$ if $f_2=x_7x_8$. Since $(\{e_1\}\cup \mathcal{F}_2)\subseteq E(C)$, {  we have $e_1,f_2\in E(C)$.} Let $P_0$ be the $(x_1,w)$-path contained in $C$ with $f_2\notin E(P_0)$. Let $P\subseteq P_0$ be the $(u,v)$-path where $u\in \{V(P_0)\cap (V(D_1)\cup V(D_2)\}\setminus V(P_3)$ and $v\in V(P_0)\cap V(P_3)$ such that $d_{P_0}(u,v)$ is as small as possible.
Since $f_1\notin E(C)$ and $f_2\notin E(P_0)$, we have $P\neq f_1, f_2$. By the choice of $u$ and $v$, we know that $V(P)\cap (V(D_1)\cup V(D_2))=\{u,v\}$. So there is a $(u,v)$-path $P$ such that $V(P)\cap (V(D_1)\cup V(D_2))=\{u,v\}$ and $P\neq f_1, f_2$, where $u\in (V(D_1)\cup V(D_2))\setminus V(P_3)$ and $v\in V(P_3)$. By {\bf Claim}, we get a contradiction. Hence $\mathcal{F}_1\subseteq\mathcal{F}_2$.

In particular, if $e_2\in E(D_1)$ and $\mathcal{F}_2\subseteq (E(D_1)\cap E(D_2)$, then we have $\mathcal{F}_1=\mathcal{F}_2$ by the symmetry of $e_1$ and $e_2$. This completes the proof of Theorem \ref{thm1}.
\end{proof}

\begin{thm}\label{mthm}
Let $G$ be a $(3,2)$-critical graph with at least five odd cycles and $H\in \mathcal{G}$ with $H\subseteq G$. Then (i) $H\notin \mathcal{G}\setminus\{\mathcal{G}_4\cup \mathcal{G}_5\}$, and (ii) $H\in \mathcal{G}_4\cup \mathcal{G}_5$ (see Figure~\ref{fig15}) satisfying that $P_2\cup P_3$ is an  even cycle in $H$.
\end{thm}

 \begin{proof}
By Proposition~\ref{prop}, there exists an $H\in\mathcal{G}$ such that $H\subseteq G$. { Let $D_1$ and $D_2$ be as stated when we introduce the definition  of $\mathcal{G}_i$ for $i\in [7]$. Clearly we have known that $D_1$ and $D_2$ are odd cycles.}
By the definition of $(3,2)$-critical, the following observation holds directly.\medskip

\noindent  {\bf Observation}: If $G$ is $(3,2)$-critical, then for any $e\in E(G)$, all odd cycles share one edge in $G-e$.

It suffices to prove the following three claims.

{\bf Claim 1.}  $H\notin\mathcal{G}_1\cup\mathcal{G}_2\cup\mathcal{G}_3$.

{\bf Proof.} Suppose $H\in\mathcal{G}_1\cup\mathcal{G}_2\cup\mathcal{G}_3$. Let $P=P_5\cup P_6$, $P=P_5$ and $P=\emptyset$ if $H\in\mathcal{G}_1$, $H\in\mathcal{G}_2$ and $H\in\mathcal{G}_3$, respectively. Then $D_1=P_1\cup P_3\cup P$ and $D_2=P_3\cup P_4$ are odd cycles.

We first claim $P_1\cup P_2$ is an even cycle. Suppose $H\in\mathcal{G}_1$ or $H\in\mathcal{G}_2$. Let $e\in E(P_5)$. By {\bf Observation}, all odd cycles share one edge in $G-e$. Since there is no edge in $(P_1\cup P_2)\cap D_2$, $P_1\cup P_2$ is an even cycle.   Suppose $H\in\mathcal{G}_3$. If $P_1\cup P_2$ is an odd cycle, then $H$ contains exactly four odd cycles { $D_1,D_2,P_1\cup P_2$ and $P_2\cup P_4$}. Since $G$ contains at least five odd cycles, there exists an edge $e\in E(G)\setminus E(H)$. By {\bf Observation}, all odd cycles share one edge in $G-e$, contradicting the fact that $E(P_1\cup P_2)\cap E(D_2)=\emptyset$. So $P_1\cup P_2$ is an even cycle. This means that $D'_1=P_2\cup P_3\cup P$ is  an odd cycle since $D_1$ is an odd cycle. For any  $e_1,e_2\in E(P_1)\cup E(P_2)$, we have $\emptyset\neq\mathcal{F}_1,\mathcal{F}_2\subseteq E(P_3)$. { Without loss of generality, suppose that $e_1,e_2\in E(P_1)$, or $e_1\in E(P_1)$ and $e_2\in E(P_2)$. If $e_1,e_2\in E(P_1)$, then $D_1,D_2$ and $e_1,e_2$ satisfy the conditions in Theorem~\ref{thm1} since $D_1\cap D_2=P_3$. By the symmetry of $e_1$ and $e_2$, we have $\mathcal{F}_1=\mathcal{F}_2\neq \emptyset$ by Theorem~\ref{thm1}. If $e_1\in E(P_1)$ and $e_2\in E(P_2)$, we have $D_1,D_2,e_1,e_2$ and $D_1',D_2,e_1,e_2$ satisfy the conditions in Theorem~\ref{thm1} since $D_1\cap D_2=D_1'\cap D_2=P_3$. In this case again we have $\mathcal{F}_1=\mathcal{F}_2\neq \emptyset$ by Theorem~\ref{thm1}. Therefore, for any $e_1,e_2\in E(P_1\cup P_2)$, $\mathcal{F}_1=\mathcal{F}_2$.} For any edge $f\in \mathcal{F}_1=\mathcal{F}_2$ and $e\in E(P_1)\cup E(P_2)$, we have $\chi(G-f)=3$ and $\chi(G-\{e,f\})=2$ {  as $G$ is $(3,2)$-critical}. So there is at least one odd cycle $C$ in $G-f$ such that $e\in E(C)$. By the arbitrariness of $e$, we have $P_1\cup P_2\subseteq C$. Hence $C=P_1\cup P_2$  is an odd cycle, contradicting the fact that $P_1\cup P_2$ is an even cycle.\qed

{\bf Claim 2.}  $H\notin\mathcal{G}_6\cup\mathcal{G}_7$.

{\bf Proof.} Suppose $H\in\mathcal{G}_6$. Let $P_1,P_2,P_3$ be the $(x,u)$-path, $(u,y)$-path, $(y,x)$-path of $D_1$ and $P_4,P_5$ be the $(y,v)$-path, $(v,x)$-path of $D_2$ in a counter clockwise direction, respectively,  as shown in Figure~\ref{fig15}. Then $D_1=P_1\cup P_2\cup P_3$ and $D_2=P_4\cup P_5\cup P_3$. Let $P_6$ denote the $(u,v)$-path that is internally disjoint from $D_1\cup D_2$. Since both $D_1$ and $D_2$ are odd cycles, {  $P_1\cup P_5\cup P_4\cup P_2$ is an even cycle. So }the two cycles  $D_3=P_6\cup P_2\cup P_4$ and $D_4=P_6\cup P_5\cup P_1$ have the same parity. If $D_3$ and $D_4$ are  odd cycles, then $H$ contains exactly four odd cycles. Since $G$ contains at least five odd cycles, there exists an edge $e\in E(G)\setminus E(H)$. By {\bf Observation}, all odd cycles share one edge in $G-e$, contradicting the fact that $E(D_1)\cap E(D_2) \cap E(D_3) \cap E(D_4)=\emptyset$. So $D_3$ and $D_4$ are  even cycles. Then $D_5=P_5\cup P_6\cup P_2\cup P_3$ and $D_6=P_1\cup P_6\cup P_4\cup P_3$ are odd cycles. For $e_1=w_1v\in E(P_6)$, $e_2=w_2v\in E(P_5)$ and $e_3=w_3v\in E(P_4)$, let $\mathcal{F}_i=\{f_i\in E(G)\mid \chi(G-\{e_i,f_i\})=2\}$ for $i\in[3]$. Then we have  $\emptyset\neq\mathcal{F}_1\subseteq E(P_3)$,  $\emptyset\neq\mathcal{F}_2\subseteq E(P_1\cup P_3)$ and $\emptyset\neq\mathcal{F}_3\subseteq E(P_2\cup P_3)$. Note that $D_1\cap D_6=P_1\cup P_3$ with $e_1\in E(D_6)\setminus E(D_1)$ and $e_2\notin E(D_1)\cup E(D_6)$, and $D_1\cap D_5=P_2\cup P_3$ with $e_1\in E(D_5)\setminus E(D_1)$ and $e_3\notin E(D_1)\cup E(D_5)$. { We have that $D_6,D_1, e_1,e_2$ and $D_5,D_1, e_1,e_3$ satisfy the conditions in Theorem~\ref{thm1}. Therefore}
by Theorem~\ref{thm1}, we have $\emptyset\neq \mathcal{F}_1\subseteq \mathcal{F}_2$ and $\emptyset\neq \mathcal{F}_1\subseteq \mathcal{F}_3$. For any edge $f\in \mathcal{F}_1$ and $e_i$ ($i\in[3]$), we have $\chi(G-f)=3$ and $\chi(G-\{e_i,f\})=2$. So there is at least one odd cycle $C$ in $G-f$ such that $e_1,e_2,e_3\in E(C)$. Then the degree of $v$ in $C$ is three, a contradiction.

Suppose $H\in\mathcal{G}_7$. Let $P_1,P_2,P_4,P_5$ be the $(x,u)$-path, $(u,y)$-path, $(y,v)$-path, $(v,x)$-path of $D_1$ and $P_6$ be the $(y,x)$-path of $D_2$ in a counter clockwise direction, respectively, as shown in Figure~\ref{fig15}. Then  $D_1=P_1\cup P_2\cup P_4\cup P_5$ and $D_2=P_4\cup P_5\cup P_6$. Let $P_3$ denote the $(u,v)$-path that is internally disjoint from $D_1\cup D_2$. Since both $D_1$  and $D_2$  are odd cycles,
we have $P_1\cup P_2\cup P_6$ is an even cycle and either $P_1\cup P_3\cup P_5$ or $P_2\cup P_3\cup P_4$ is an odd cycle. Without loss of generality, we assume $D_3=P_2\cup P_3\cup P_4$ is an odd cycle. Then $D_4=P_3\cup P_1\cup P_6\cup P_4$ is an odd cycle. For $e_1=v_1x\in E(P_1)$,  $e_2=v_2x\in E(P_5)$ and  $e_3=v_3x\in E(P_6)$, let $\mathcal{F}_i=\{f_i\in E(G)\mid \chi(G-\{e_i,f_i\})=2\}$ for $i\in[3]$. Then we have $\emptyset\neq\mathcal{F}_1\subseteq E(P_4)$, $\emptyset\neq\mathcal{F}_2\subseteq E(P_3\cup P_4)$, and $\emptyset\neq\mathcal{F}_3\subseteq E(P_2\cup P_4)$. Note that $D_3\cap D_4=P_3\cup P_4$ with $e_1\in E(D_4)\setminus E(D_3)$ and $e_2\notin E(D_3)\cup E(D_4)$, and $D_1\cap D_3=P_4\cup P_2$ with $e_1\in E(D_1)\setminus E(D_3)$ and $e_3\notin E(D_1)\cup E(D_3)$. { We have that $D_3,D_4,e_1,e_2$ and $D_3,D_1,e_1,e_3$ satisfy the conditions in Theorem~\ref{thm1}. Therefore,} we have $\emptyset\neq \mathcal{F}_1\subseteq \mathcal{F}_2$ and $\emptyset\neq \mathcal{F}_1\subseteq \mathcal{F}_3$ by Theorem~\ref{thm1}. For any edge $f\in \mathcal{F}_1$ and $e_i$ ($i\in[3]$), we have $\chi(G-f)=3$ and $\chi(G-\{e_i,f\})=2$. So there is at least one odd cycle $C$ in $G-f$ such that $e_1,e_2,e_3\in E(C)$. Then the degree of $x$ in $C$ is three, a contradiction.\qed

%
{\bf Claim 3.}  $H\in\mathcal{G}_4\cup\mathcal{G}_5$ {  and $P_2\cup P_3$ is an even cycle in $H$}.

{\bf Proof.}  Suppose $H\in\mathcal{G}_4\cup\mathcal{G}_5$.
Let $D_3=P_2\cup P_3$. If $D_3$ is an odd cycle, then  $\mathcal{G}_4=\mathcal{G}_2$ and $\mathcal{G}_5=\mathcal{G}_1$. By {\bf Claim 1}, we know $G$ contains no graph from $\mathcal{G}_1\cup\mathcal{G}_2$ as a subgraph. Therefore if $H\in\mathcal{G}_4\cup\mathcal{G}_5$, then $D_3=P_2\cup P_3$ is an even cycle.\qed

The proof is thus complete.
\end{proof}
\section{Proof of Theorem~\ref{mainthm}}

Let $\{H_i\mid i\in[k]\}$ $(k\geq3)$ be a family of graphs satisfying the following three conditions: (1) $H_i$ is an even cycle or  a path for any $i\in[k]$; (2) there are at least two  even cycles and at least one path; (3) for $i\in[k]$, if $H_i$ is a path, then $H_{i-1}$ and $H_{i+1}$ are not paths, where the subscripts are taken cyclically modulo $k$. We define the family $\mathcal{E}'$, which is obtained from the disjoint union of $k$ graphs $H_1, H_2,\ldots, H_k$ as follows. For each $i\in[k]$, let $x_i$ and $y_i$ be any two distinct vertices of $H_i$ if $H_i$ is an even cycle, and be the two endpoints of $H_i$ if $H_i$  is a path, where we only require that $\sum_{i=1}^kd_{H_i}(x_i,y_i)$ is odd. A graph $H\in \mathcal{E}'$ is obtained by identifying $y_i$ and $x_{i+1}$ for $i\in[k-1]$, and identifying $y_k$ and $x_1$. { Similarly, denote the even cycle $C_{2n_i}$ in a graph $F\in \mathcal{E}$ by $H_i$ for $i\in [k]$. We have the following lemma.}
\begin{lem}\label{mainlem}
If we add an open ear $P$ with endpoints $u$ and $v$ to $F\in\mathcal{E}'\cup \mathcal{E} $, then $F+P$ contains a graph from $\mathcal{G}_1\cup\mathcal{G}_2\cup\mathcal{G}_3\cup\mathcal{G}_6\cup\mathcal{G}_7$ as a subgraph except when
$u$ and $v$ belong to the same $H_i$ {  which is a path} and the new cycle generated by $H_i\cup P$ is an even cycle.

\end{lem}
\begin{proof}
Let $F\in\mathcal{E}'\cup \mathcal{E} $ and $P$ be an open ear of $F$ with endpoints $u$ and $v$.

 {\bf Case 1.} $u,v \in V(H_i)$ and $H_i$  is an even cycle.

{  Denote by $P_1$ and $P_2$ the two internally disjoint $(x_i,y_i)$-paths in $H_i$. By the construction of $F$, there exists an $(x_i,y_i)$-path $P_3$ in $F$ which is internally disjoint with $P_1$ and $P_2$ such that $D_1=P_1\cup P_3$ and $D_2=P_2\cup P_3$ are odd cycles. If $u,v\in V(P_1)$, then $D_1\cup D_2\cup P\in \mathcal{G}_1 \cup \mathcal{G}_2 \cup \mathcal{G}_3$. If $u\in V(P_1)$ and $v\in V(P_2)$, then $D_1\cup D_2\cup P\in \mathcal{G}_6$.} Therefore, $D_1\cup D_2\cup P\in \mathcal{G}_1 \cup \mathcal{G}_2 \cup \mathcal{G}_3\cup \mathcal{G}_6$.

{\bf Case 2.} $u\in V(H_i)$ and $v\in V(H_j)$ ($i<j$), where $H_i$ and $H_j$ are even cycles.

If $u$ or $v\in V(H_i)\cap V(H_j)$, then it can be reduced to {\bf Case 1}. So we assume $u,v\notin V(H_i)\cap V(H_j)$. {  Denote by $P_1$ and $P_2$ the two internally disjoint $(x_i,y_i)$-paths in $H_i$, $P_3$ and $P_4$ the two internally disjoint $(x_j,y_j)$-paths in $H_j$. 
By the construction of $F$, there exist $(y_i,x_j)$-path $P_5$ and $(y_j,x_i)$-path $P_6$ such that (i) $P_s$ and $P_t$ are internally disjoint for $s\neq t$ and $s,t\in [6]$, (ii) $D_1=P_1\cup P_5\cup P_3\cup P_6$ and $D_2=P_2\cup P_5\cup P_3\cup P_6$ are odd cycles. Since $H_i$ and $H_j$ are even cycles, we have $P_1\cup P_5\cup P_4\cup P_6$ is also an odd cycle. Without loss of generality, let $u\in V(P_1)$ and $v\in V(P_3)$.}
Suppose $u\notin \{x_i, y_i\}$. Since $D_1=P_1\cup P_5\cup P_3\cup P_6$ and $D_2=P_2\cup P_5\cup P_3\cup P_6$ are odd cycles, $D_1\cup D_2\cup P\in \mathcal{G}_7$. Suppose $v\notin\{x_j, y_j\}$. Since  $D_1=P_1\cup P_5\cup P_3\cup P_6$ and  $D_2=P_1\cup P_4\cup P_5\cup P_6$ are odd cycles, $D_1\cup D_2\cup P\in \mathcal{G}_7$. Suppose $u\in\{x_i, y_i\}$ and $v\in \{x_j, y_j\}$. Since $k\geq3$, at least one of $P_5$ and $P_6$ is not an isolated vertex. Without loss of generality, we assume $P_5$ is not an isolate vertex. It suffices to consider the following two subcases.

{\bf Subcase 2.1.} $u=x_i$ and $v=x_j$.

{  Suppose that $y_j\neq x_i$, otherwise it can be reduced to {\bf Case 1}.} If $P\cup P_5\cup P_1$ is an odd cycle, then let  $D_1=P_1\cup P_5\cup P_3\cup P_6$ and $D_2=P\cup P_5\cup P_1$. Thus $D_1\cup D_2\cup P_4\in\mathcal{G}_2$. If $P\cup P_5\cup P_1$ is an even cycle, then $P\cup P_3\cup P_6$ is an odd cycle since $P_1\cup P_5\cup P_3\cup P_6$ is an odd cycle. Let $D_1=P_1\cup P_5\cup P_3\cup P_6$ and $D_2=P\cup P_3\cup P_6$. Thus, $D_1\cup D_2\cup P_2\in\mathcal{G}_2$.

{\bf Subcase 2.2.} $u=y_i$ and $v=x_j$.

For the case that $P_5$ is some $H_i$ of $F$ and $H_i$ is a path, we will consider it in {\bf Case 5}. So we assume there is an even cycle $H_s$ ($s\in[k]$ and $s\neq i,j$) of $F$ such that $H_s\cap P_5$ is a nontrivial path.  Let $P'$ be the $(u',v')$-path with $u',v'\in V(P_5)$, $P'\subseteq H_s$ and $P'\not\subseteq P_5$.
If $P\cup P_5$ is an odd cycle, then let $D_1=P_1\cup P_5\cup P_3\cup P_6$ and $D_2=P\cup P_5$. Thus $D_1\cup D_2\cup P_4\in\mathcal{G}_2$. If $P\cup P_5 $ is an even cycle, then let $D_1=P_1\cup P_5\cup P_3\cup P_6$ and $D_2=P_1\cup P\cup P_3\cup P_6$. {  If $\{u',v'\}\cap\{y_i,x_j\}=\emptyset$, then $D_1\cup D_2\cup P'\in \mathcal{G}_1$. If $|\{u',v'\}\cap\{y_i,x_j\}|=1$, then $D_1\cup D_2\cup P'\in \mathcal{G}_2$. Otherwise, $D_1\cup D_2\cup P'\in \mathcal{G}_3$.} Thus, $D_1\cup D_2\cup P'\in \mathcal{G}_1 \cup \mathcal{G}_2{ \cup\mathcal{G}_3}$.

{\bf Case 3.} $u\in V(H_i)$ and $v\in V(H_j)$ ($i<j$), where $H_i$ is an even cycle and $H_j$ is a path.

If $v\in V(H_i)\cap V(H_j)$, then it can be reduced to {\bf Case 1}. If $u\in V(H_i)\cap V(H_j)$, then we will consider it in {\bf Case 5}. So we assume $u,v\notin V(H_i)\cap V(H_j)$. { Denote by $P_1$ and $P_2$ the two internally disjoint $(x_i,y_i)$-paths in $H_i$. By the construction of $F$, there are $(y_i,x_j)$-path $P_3$ and $(y_j,x_i)$-path $P_4$ such that $P_1$, $P_2$, $P_3$ and $P_4$ are four internally disjoint paths in $F$, $P_1\cup P_3\cup H_j\cup P_4$ and $P_2\cup P_3\cup H_j\cup P_4$ are odd cycles.} Suppose $u\notin \{x_i, y_i\}$. Since $D_1=P_1\cup P_3\cup H_j\cup P_4$ and $D_2=P_2\cup P_3\cup H_j\cup P_4$ are odd cycles, $D_1\cup D_2\cup P\in \mathcal{G}_7$. Suppose $u\in \{x_i, y_i\}$.
Since $k\geq3$, we have $x_i\neq y_j$ or $x_j\neq y_i$. Without loss of generality, we assume $x_i\neq y_j$. Since $u,v\notin V(H_i)\cap V(H_j)$, we may let $u=x_i$ in the following.
Since $H_j$ is a path, by the definition of $\mathcal{E}'$, $H_{j-1}$ and $H_{j+1}$ are not paths, thus there is an even cycle $H_s$ ($s\in[k]$ and $s\neq i,j$) of $F$ such that $H_s\cap P_4$ is a path.   Let $P'$ be the $(u',v')$-path with $u',v'\in V(P_4)$, $P'\subseteq H_s$ and $P'\not\subseteq P_4$.
Let $D_1=P_1\cup P_3\cup H_j\cup P_4$. The two  internally disjoint $(u,v)$-paths of $D_1$ are denoted by $Q_1$ and $Q_2$, where $P_1\subseteq Q_1$. Obviously, $Q_1\cup P$ or $Q_2\cup P$ is an odd cycle. If $D_2=Q_1\cup P$ is an odd cycle, {  take the position of $v\in V(H_j)$ and the intersection of $\{u',v'\}$ and $\{x_i,y_j\}$ into consideration,} then $D_1\cup D_2\cup P'\in \mathcal{G}_1 \cup \mathcal{G}_2\cup \mathcal{G}_3$. If $D_2=Q_2\cup P$ is an odd cycle, then $D_1\cup D_2\cup P_2\in \mathcal{G}_2\cup \mathcal{G}_3$.

{\bf Case 4.} $u\in V(H_i)$ and $v\in V(H_j)$ ($i<j$), where $H_i$  and $H_j$ are two paths.

{  By the construction of $F$, there are $(y_i,x_j)$-path $P_1$ and $(y_j,x_i)$-path $P_2$ such that $P_1$, $P_2$, $H_i$ and $H_j$ are four internally disjoint paths of $F$ and $H_i\cup P_1\cup H_j\cup P_2$ is an odd cycle}. By the definition of $\mathcal{E}'$, we have $x_i\neq y_j$, $y_i\neq x_j$, $u,v\notin V(H_i)\cap V(H_j)$ and there are two different even cycles $H_s$ and $H_t$ ($s,t\in[k]$ and $s,t\neq i,j$) of $F$ such that $H_s\cap P_1$ is a path and $H_t\cap P_2$ is a path.  Let $P'_1$ be the $(u'_1,v'_1)$-path with $u'_1,v'_1\in V(P_1)$, $P'_1\subseteq H_s$ and $P'_1\not\subseteq P_1$. Let $P'_2$ be the $(u'_2,v'_2)$-path with $u'_2,v'_2\in V(P_2)$, $P'_2\subseteq H_t$ and $P'_2\not\subseteq P_2$.
Let $D_1=H_i\cup P_1\cup H_j\cup P_2$. The two  internally disjoint $(u,v)$-paths of $D_1$ are denoted by $Q_1$ and $Q_2$, where $P_1\subseteq Q_1$. Obviously, $Q_1\cup P$ or $Q_2\cup P$ is an odd cycle. If $D_2=Q_1\cup P$ is an odd cycle, then $D_1\cup D_2\cup P'_2\in \mathcal{G}_1 \cup \mathcal{G}_2\cup \mathcal{G}_3$. If $D_2=Q_2\cup P$ is an odd cycle, then $D_1\cup D_2\cup P'_1\in \mathcal{G}_1\cup\mathcal{G}_2\cup \mathcal{G}_3$.

{\bf Case 5.} $u,v \in V(H_i)$, where $H_i$ is a path and the new cycle generated by $H_i\cup P$ is an odd cycle.

{ There are two internally disjoint $(u,v)$-paths in $F$, say $P_1$ and $P_2$ such that $P_1\cup P_2$ is an odd cycle and} $P\cup P_1$ is the new odd cycle generated by $H_i\cup P$. By the definition of $\mathcal{E}'$, there is an even cycle $H_s$ ($s\in[k]$ and $s\neq i$) of $F$ such that $H_s\cap P_2$ is a path.  Let $P'$ be the $(u',v')$-path with  $u',v'\in V(P_2)$, $P'\subseteq H_s$ and $P'\not\subseteq P_2$. Let $D_1=P_1\cup P_2$ and $D_2=P_1\cup P$. {  Since $H_i$ is a path, we have $|\{u',v'\}\cap \{x_i,y_i\}|\le1$ by the construction of $F$.} Therefore, $D_1\cup D_2\cup P'\in \mathcal{G}_1\cup\mathcal{G}_2$.

The proof of Lemma \ref{mainlem} is thus complete.
 \end{proof}

{ {\bf Proof of Theorem~\ref{mainthm}.}} Let $G$ be a $(3,2)$-critical graph without isolated vertices with at least five odd cycles. By Theorem~\ref{mthm}, $G$ contains a graph $H$ from $\mathcal{G}_4\cup\mathcal{G}_5$ (see Figure~\ref{fig15}) as a subgraph and $P_2\cup P_3$ is an even cycle in $H$. Obviously, $H\in \mathcal{E}'$.
By Lemma~\ref{conn},  $G$ is nonseparable.
Then by Proposition \ref{ear}, $G$ has a decomposition as $D_0, Q_1,\ldots, Q_\ell$ and $G=D_0\cup Q_1\cup\ldots\cup Q_\ell$, where $D_0=H\in\mathcal{G}_4\cup\mathcal{G}_5$, $Q_1$ is an open ear of $D_0$, and $Q_i$ is an open ear of $D_{i-1}=D_{i-2}\cup Q_{i-1}$ for $2\leq i\leq \ell$.
By combining Lemma \ref{mainlem} with Theorem~\ref{mthm},  $D_i=D_{i-1}\cup Q_i\in \mathcal{E}'\cup \mathcal{E}$ for $i\in [\ell-1]$. Note that if $F\in\mathcal{E}'$, then there exists an edge $e\in E(F)$ such that $\chi(F-e)=2$. Since $G$ is a $(3,2)$-critical graph, we have $G=D_\ell=D_{\ell-1}\cup Q_\ell\in \mathcal{E}$.

 Thus, we complete the proof.\hfill\vrule height3pt width6pt depth2pt\bigskip

\noindent{\bf Acknowledgements.} The authors would like to thank the two referees very much for their valuable suggestions and comments.

Lei was partially supported by the National
Natural Science Foundation of China (No. 12001296) and the Fundamental Research Funds
for the Central Universities, Nankai University (No. 63201163). Shi was partially supported by the National
Natural Science Foundation of China (No. 11922112), the Natural Science Foundation of Tianjin (Nos. 20JCJQJC00090 and 20JCZDJC00840) and the Fundamental Research Funds for the Central Universities, Nankai University (No. 63206034). Wang was partially supported by the National Natural Science Foundation of China (No. 12071048) and the Science
and Technology Commission of Shanghai Municipality (No. 18dz2271000).

\end{document}